\documentclass[preprint]{imsart}

\pdfoutput=1
\RequirePackage[OT1]{fontenc}
\usepackage{amsthm, amsmath, natbib, amssymb, enumitem}
\usepackage{empheq}
\RequirePackage[colorlinks,citecolor=blue,urlcolor=blue]{hyperref}
\usepackage{pgfplots}
\pgfplotsset{compat=1.18}
\usepackage{mathtools} 
\mathtoolsset{showonlyrefs}
\usepackage[table]{xcolor}
\usepackage{array}

\startlocaldefs
\numberwithin{equation}{section}
\theoremstyle{plain}
\newtheorem{theorem}{Theorem}[section]
\newtheorem{lemma}{Lemma}[section]

\theoremstyle{remark}

\newcommand{\JSJS}{{ \mathrm{JS} }}

\newcommand{\EE}{\mathrm{E}}
\newcommand{\rd}{\mathrm{d}}

\endlocaldefs
\usepackage[margin=35truemm]{geometry}
\usepackage{booktabs}
\begin{document}
\begin{frontmatter}
\title{Improving the James--Stein estimator via $q$-number shrinkage}
\runtitle{A New James--Stein Estimator}

\begin{aug}
\author{\fnms{Yuzo} \snm{Maruyama}\thanksref{addr1,t1}\ead[label=e1]{maruyama@math.s.chiba-u.ac.jp}}
\and
\author{\fnms{Akimichi} \snm{Takemura}\thanksref{addr2,t2}\ead[label=e2]{a-takemura@biwako.shiga-u.ac.jp}}

\runauthor{Y.~Maruyama and A.~Takemura}

\address[addr1]{Chiba University \printead{e1}}
\address[addr2]{Shiga University \printead{e2}}

\thankstext{t1}{supported by JSPS KAKENHI Grant Number 22K11933}
\thankstext{t2}{supported by JSPS KAKENHI Grant Number 24K14852}
\end{aug}

\begin{abstract}
For estimating the mean vector of a $p$-variate normal distribution ($p \ge 3$) under quadratic loss,
the positive-part James-Stein estimator dominates the original Jame-Stein estimator
though it suffers from a non-smooth thresholding boundary.
In this paper, we first express the shrinkage factor of the positive-part estimator 
using an infinite series. 
Truncating this series at a finite degree $m$ allows us to 
interpret the factor as a special case of $q$-numbers. 
Leveraging these $q$-numbers parameterized by a positive real index $m$,
we propose a novel smooth estimator.
We prove that the proposed estimator dominates 
the original James-Stein estimator for any dimension $p \ge 3$,
provided the truncation parameter satisfies $m \ge 0.95\sqrt{p-2}$.
\end{abstract}



\end{frontmatter}

\section{Introduction and a main result}
\label{sec:intro}
Let $ X$ have a $p$-variate normal distribution 
$ \mathcal{N}_{p} (\theta, I_{p}) $. 
We consider the problem of estimating the mean vector $\theta$ under 
the quadratic loss function $\| \hat{\theta} - \theta \|^2$.
Then the risk function of an estimator $ \hat{\theta}(X)$ is
$ R(\theta,\hat{\theta})=\EE_\theta\bigl[ \| \hat{\theta}(X) - \theta \|^2\bigr]$.

The usual unbiased estimator, $ X $, has constant risk $p$ and is
minimax for $p\in\mathbb{N}$. 
\cite{Stein-1956} showed that, for $p\geq 3$,
there exist estimators dominating the usual estimator $X$.
\cite{James-Stein-1961}
found an explicit dominating procedure
\begin{equation}\label{JS}
\hat{\theta}_{\JSJS}(X)=\left( 1- \frac{p-2}{\| X \|^2}\right)X,
\end{equation}  
called the James--Stein estimator \citep{Efron-2024}. 
Further, as shown in \cite{Baranchik-1964}, the James--Stein estimator is inadmissible,
as the positive-part estimator
\begin{equation}\label{JSPP}
\hat{\theta}_{\JSJS}^{+}(X)=\max\left(0, 1- \frac{p-2}{\| X \|^2}\right)X
\end{equation}  
dominates $\hat{\theta}_{\JSJS}$.
The James--Stein positive-part estimator 
can be written as 
\begin{equation}
\hat{\theta}_{\JSJS}^{+}(X)= \left(1 - \frac{\phi_+(\|X\|^2)}{\|X\|^2}\right)X,
\end{equation}
where
\begin{equation}
 \phi_+(w) = \min(w, p-2) = (p-2)\left(1 - \frac{1}{\sum_{i=0}^\infty \{w/(p-2)\}^i}\right).\label{jspp.infinite}
\end{equation}

In this paper, we start with 
replacing the infinite sum in \eqref{jspp.infinite} with a finite sum and 
consider the shrinkage estimator 
$\hat{\theta}(X;m)=\left(1 - \phi(\|X\|^2;m)/\|X\|^2\right)X$ based on
\begin{equation}\label{phi.m}
\phi(w;m) = (p-2)\left(1 - \frac{1}{\sum_{i=0}^m \{w/(p-2)\}^i}\right).
\end{equation}
Further $\phi(w;m) $ 
can be extended to an arbitrary positive real number $m$, which is not necessarily an integer, as
\begin{equation}\label{q-number-0}
\phi(w;m) = (p-2)\left(1-\frac{1}{[m+1]_{w/(p-2)}}\right),
\end{equation}
where $[n]_q$ denotes the $q$-number of $n$, that is,
\begin{equation}
 [n]_q=
\begin{cases}
 n & \text{for }q=1\\
\displaystyle \frac{q^n-1}{q-1}& \text{for }q\neq 1.
\end{cases}
\end{equation}
The $q$-number formalism is predominantly used in mathematical physics, 
combinatorics, and representation theory.
Then, for $u=w/(p-2)$,
we have
\begin{equation}\label{q-number}
1 - \frac{\phi(w;m)}{w} = 
1-\frac{1}{u}\Bigl(1-\frac{1}{[m+1]_u}\Bigr)
=\frac{u[m+1]_u-[m+1]_u+1}{u[m+1]_u}.
\end{equation}
For $u=1$, $ 1 - \phi(w;m)/w=1/(m+1)$ follows. For $u\neq 1$, we have
\begin{equation}
 \frac{u[m+1]_u-[m+1]_u+1}{u[m+1]_u}=\frac{u^{m+1}}{u[m+1]_u}=\frac{1}{[m+1]_{1/u}}.
\end{equation}
Therefore we have
\begin{equation}\label{1-phi/w}
 1 - \frac{\phi(w;m)}{w} = \frac{1}{[m+1]_{(p-2)/w}}\in [0,1],
\end{equation}
which nicely avoids the need to take the positive part.
With the choice of $m \geq 0.95 \sqrt{p-2}$, 
the main theorem of this paper is stated as follows.

\begin{theorem}\label{thm:main}
Assume $p\geq 3$. Then the shrinkage estimator
\begin{equation}\label{main.form}
\hat{\theta}(X;m)=
\begin{cases}
\displaystyle \frac{1}{[m+1]_{(p-2)/\|X\|^2}}X & \text{for a positive non-integer }m\\[7pt]
\displaystyle \frac{1}{1+\sum_{i=1}^m \{(p-2)/\|X\|^2\}^i}X & \text{for a positive integer }m
\end{cases}
\end{equation}
dominates the James--Stein estimator if $m\geq 0.95\sqrt{p-2}$.
\end{theorem}

One of our motivations for proposing a new type of James--Stein estimator stems 
from the block James--Stein method known in the field of wavelet analysis
(See \cite{johnstone-2019-book, Candes-2006}).
The interpretation is that very low-frequency components are left untouched, 
intermediate-frequency components are shrunk toward zero, 
and high-frequency components are discarded. 
For the intermediate-frequency components, 
the non-negative part estimator given in \eqref{JSPP} is employed. 
However, within the intermediate-frequency range, 
sparsity in the sense that entire blocks become zero is not expected. 
Rather, it is more natural to estimate the components up to the intermediate-frequency 
region, just before the high-frequency components that are discarded, as non-zero. 
Therefore, 
our simple and tuning-free estimator, $\hat{\theta}(X;m)$ \eqref{main.form}, 
for $m=\sqrt{p-2}$ or $m=\lceil \sqrt{p-2} \;\rceil$, 
the smallest integer greater than or equal to $\sqrt{p-2}$,
is better suited for use in the block James--Stein method.

\section{Proof}
\label{sec:proof}
\subsection{The risk difference at the origin}
For the shrinkage estimator $\hat{\theta}_{\phi}(X)= ( 1- \phi(\| X \|^2)/\| X \|^2)X$,
\cite{Stein-1974} expressed the risk as 
$ R(\theta,\hat{\theta}_\phi) =\EE_\theta\bigl[\hat{R}_\phi(\|X\|^2)\bigr]$,
where 
\begin{equation}\label{stein.identity.2}
 \hat{R}_\phi(w)=p+\frac{\phi(w)}{w}\left\{\phi(w)-2(p-2)\right\}-4\phi'(w).
\end{equation}
Then
the difference of the risks, $R(\theta,\hat{\theta}_{\JSJS}(X))- R(\theta,\hat{\theta}(X;m))$, 
is 
\begin{equation}\label{risk.diff.1}
R(\theta,\hat{\theta}_{\JSJS}(X))- R(\theta,\hat{\theta}(X;m))
=\EE_\theta\bigl[\hat{R}_{\JSJS}(\|X\|^2)-\hat{R}(\|X\|^2;m)\bigr]
\end{equation}
where
\begin{equation}\label{Phi.1}
\begin{split}
 \hat{R}_{\JSJS}(w)-\hat{R}(w;m)
&=4\phi'(w;m)-\frac{\{p-2-\phi(w;m)\}^2}{w} \\
&=\bigl(4\Psi(w;m)-1\bigr)\frac{\{p-2-\phi(w;m)\}^2}{w},
\end{split} 
\end{equation}
and where $\Psi(w;m)$ in \eqref{Phi.1} is given by
\begin{equation}
 \Psi(w;m) = \frac{w\phi'(w;m)}{\{p-2-\phi(w;m)\}^2}
=\frac{w}{(p-2)^2}\frac{\rd}{\rd w}[m+1]_{w/(p-2)}.
\end{equation}
In Part \ref{lem:q-number-1} of Lemma \ref{lem:q-number} below, $ \Psi(w;m)$ is strictly increasing in $w$.
Note the distribution of $W=\|X\|^2$ follows a noncentral chi-square distribution 
with $p$ degrees of freedom and noncentrality parameter $\lambda = \|\theta\|^2$, 
whose probability density function is denoted by $f_p(w; \lambda)$.
The noncentral chi-square distribution possesses a monotone likelihood ratio; 
that is, $f_p(w; \lambda)/f_p(w; 0)$ is monotone increasing in $w$.
Then, following the method of \cite{Maruyama-Takemura-2025}, we have
\begin{align}
 R(\theta,\hat{\theta}_{\JSJS}(X))- R(\theta,\hat{\theta}(X;m))
&=\int_0^\infty
\bigl(4\Psi(w;m)-1\bigr)\frac{\{p-2-\phi(w;m)\}^2}{w}f_p(w;\lambda)\rd w\label{difference.0}\\
&=\int_0^\infty
\bigl(4\Psi(w;m)-1\bigr)\frac{f_p(w;\lambda)}{f_p(w;0)}\frac{\{p-2-\phi(w;m)\}^2}{w}f_p(w;0)\rd w
\\
&\geq
\int_0^\infty
\bigl(4\Psi(w;m)-1\bigr)\frac{\{p-2-\phi(w;m)\}^2}{w}f_p(w;0)\rd w
 \\
&\quad\times
\frac{\int_0^\infty\{f_p(w;\lambda)/f_p(w;0)\}\{\{p-2-\phi(w;m)\}^2/w\}f_p(w;0)\rd w}
{\int_0^\infty\{\{p-2-\phi(w;m)\}^2/w\}f_p(w;0)\rd w}\\
&=\frac{\EE_\lambda[\{p-2-\phi(W;m)\}^2/W]}
{\EE_{\lambda=0}[\{p-2-\phi(W;m)\}^2/W]}\;\Delta(p;m),
\end{align}
where the inequality follows from the fact that $\Psi(w;m)$ is monotone increasing in $w$
and where
\begin{equation}\label{difference_origin}
\Delta(p;m)\coloneqq R(0,\hat{\theta}_{\JSJS}(X))- R(0,\hat{\theta}(X;m)).
\end{equation}
By \eqref{difference.0}, in order to prove Theorem \ref{thm:main}, it suffices to verify that 
\begin{equation}\label{Dpm0.94}
\Delta(p;m)\geq 0,\quad\text{for }p\geq 3\text{ and } m\geq 0.95\sqrt{p-2}.
\end{equation}

For $\Delta(p;m)$ given by \eqref{difference_origin}, we have
\begin{align}
 R(0,\hat{\theta}_{\JSJS}(X))&=p-\EE_{\lambda=0}\Bigl[\frac{(p-2)^2}{W}\Bigr]=p-(p-2)=2,\\
R(0,\hat{\theta}(X;m))&= \EE_{\lambda=0}\Bigl[\Bigl(1 - \frac{\phi(W;m)}{W}\Bigr)^2W\Bigr],
\shortintertext{and hence}
\Delta(p;m)&=2 - \EE_{\lambda=0}\Bigl[\Bigl(1 - \frac{\phi(W;m)}{W}\Bigr)^2W\Bigr].\label{DELTApm}
\end{align}
Recall $1-\phi(w;m)/w=1/[m+1]_{(p-2)/w}$ as in \eqref{1-phi/w}.
Then we have
\begin{equation}\label{mm1}
\Bigl(1-\frac{\phi(w;m)}{w}\Bigr)\in[0,1]\quad\text{and}\quad
 \frac{\partial}{\partial m}\Bigl(1-\frac{\phi(w;m)}{w}\Bigr)\leq 0
\end{equation}
for all $w\geq 0$,
where the inequality follows from Part \ref{lem:q-number-0} of Lemma \ref{lem:q-number} below.
Hence we have
\begin{equation}\label{R.mm1}
 \frac{\partial}{\partial m}R(0,\hat{\theta}(X;m))\leq 0.
\end{equation}
For a given $p\geq 3$, the equations \eqref{DELTApm}, \eqref{mm1} and \eqref{R.mm1} 
imply that there exists $m(p)$ satisfying
\begin{equation}
\Delta(p;m(p)) =0.
\end{equation}
The lower bound $0.95\sqrt{p-2}$ for $m$ in \eqref{Dpm0.94} originates 
from a first-order approximation of $m(p)$ as $p \to \infty$.
Details are explained in Section \ref{sec:0.94}.

\subsection{Re-expression of the risk difference}
\label{sec:reex}
Expanding and taking the expectation for $ \Delta(p;m)$ given by \eqref{DELTApm}, we obtain
\begin{align}
\Delta(p;m)
&= 2 - \EE_{\lambda=0}[W] + 2\EE_{\lambda=0}[\phi(W;m)] - \EE_{\lambda=0}[\phi^2(W;m)/W]\\
&= 2 - p + 2(p-2)\EE_{\lambda=0}\left[1 - \frac{1}{[m+1]_{W/(p-2)}}
\right]\\
&\quad - \EE_{\lambda=0}\left[\frac{(p-2)^2}{W}\left(1 - 
\frac{1}{[m+1]_{W/(p-2)}}
\right)^2\right]\\
&= (p-2)\EE_{\lambda=0}\left[\frac{1}{W}\left(\frac{2\{(p-2)-W\}}
{[m+1]_{W/(p-2)}}
- \frac{p-2}{\{
[m+1]_{W/(p-2)}
\}^2}\right)\right]\\
&= \EE\left[\frac{2(p-2-V)}{
[m+1]_{V/(p-2)}
}
- \frac{p-2}{\{
[m+1]_{V/(p-2)}
\}^2}\right]\\
&= (p-2)\EE\left[\frac{2(1-U)}{
[m+1]_U
}
- \frac{1}{\{
[m+1]_U
\}^2}\right]\\
&= (p-2)\EE[g(U;m)], \label{J}
\end{align}
where $V \sim \chi^2_{p-2}$, $U \sim \mathrm{Ga}(\{p-2\}/2, 2/\{p-2\})$, and
\begin{equation}\label{gu.0}
 g(u;m) = 
\frac{2(1-u)[m+1]_u-1}{\{[m+1]_u\}^2}=
\frac{1 - 2u^{m+1}}{\{[m+1]_u\}^2}.
\end{equation}
In what follows, we set $r = p/2 - 1$.
With the probability density function of $U \sim \mathrm{Ga}(r, 1/r)$,
\begin{equation}\label{fq}
 f_r(u) = \frac{r^r}{\Gamma(r)} u^{r-1} e^{-ru},
\end{equation}
$\Delta(2r+2;m)$ given by \eqref{J} is expressed as
\begin{equation}\label{Delta.q}
\Delta(2r+2;m) 
=2r\int_0^\infty g(u;m)f_r(u)\rd u.
\end{equation}

\subsection{Splitting the Region of Integration}
The sign of $g(u;m)$ given by \eqref{gu.0} changes from $+$ to $-$ at
\begin{equation}\label{u0}
 u_0 = (1/2)^{1/(m+1)}.
\end{equation}
Let
\begin{equation}\label{qc1m}
 r> 36, \ c_1=0.95\sqrt{2}, \text{ and } m\geq c_1\sqrt{r}.
\end{equation}
Then, using the inequality $e^x \geq 1 + x$ for $x \in \mathbb{R}$, this $u_0$ satisfies
\begin{equation}\label{uhen}
u_0 = \exp\left(-\frac{\log 2}{m+1}\right) \geq 1 - \frac{\log 2}{m+1} 
> 1 - \frac{\log 2}{c_1\sqrt{r}},
\end{equation}
where the second inequality follows from \eqref{qc1m}.
Below, we denote the right-hand side of \eqref{uhen} by
\begin{equation}\label{u1}
 u_1 = 1 - \frac{\log 2}{c_1\sqrt{r}} = 1 + \frac{c_2}{\sqrt{r}}, 
\text{ with } c_2 = -\frac{\log 2}{c_1}.
\end{equation}
From \eqref{u0}, \eqref{qc1m}, \eqref{uhen}, and \eqref{u1}, note that $0 < 1 - 6/\sqrt{r} < u_1 < u_0 < 1$.
Then, by \eqref{Delta.q}, we have
\begin{align}
\frac{\Delta(2r+2;m)}{2}
&= r\left(\int_0^{u_0} + \int_{u_0}^1 + \int_1^{1+6/\sqrt{r}} + \int_{1+6/\sqrt{r}}^\infty\right)g(u;m)f_r(u)\rd u\\
&\ge r\left(\int_{1-6/\sqrt{r}}^{u_1} + \int_{u_0}^1 + \int_1^{1+6/\sqrt{r}} + \int_{1+6/\sqrt{r}}^\infty \right)g(u;m)f_r(u)\rd u\\
&= r\int_{1-6/\sqrt{r}}^{u_1}g(u;m)f_r(u)\rd u
- r\int_{u_0}^1 |g(u;m)|f_r(u)\rd u\\ 
&\qquad - r\int_1^{1+6/\sqrt{r}} |g(u;m)|f_r(u)\rd u
- r\int_{1+6/\sqrt{r}}^\infty |g(u;m)|f_r(u)\rd u\\
&\coloneqq I_1(r;m) - I_2(r;m) - I_3(r;m) - I_4(r;m).
\end{align}
Note that the function $g(u;m)$ is given by $\eqref{gu.0}$ and
that $ g(1;m)=-1/(m+1)$. For $u\neq 1$, 
the function $g(u;m)$ can be written as
\begin{align}\label{tildegy}
g(u;m) &= \frac{(1-u)^2(1 - 2u^{m+1})}{(1-u^{m+1})^2} \\ 
&=(1-u)^2\tilde{g}(u^{m+1}), \quad\text{where}\quad \tilde{g}(y) = \frac{1 - 2y}{(1-y)^2}.
\end{align}
Thus, particularly for $I_1(r), I_2(r)$, and $I_3(r)$, 
applying the change of variables $u = 1 + x/\sqrt{r}$ yields
\begin{align}
 I_1(r;m) &= \frac{1}{r^{1/2}}\int_{-6}^{c_2}x^2 \tilde{g}((1+x/\sqrt{r})^{m+1})f_r(1+x/\sqrt{r})\rd x,\\
I_2(r;m) &= \frac{1}{r^{1/2}}\int_{\sqrt{r}(u_0-1)}^0 x^2 |\tilde{g}((1+x/\sqrt{r})^{m+1})|f_r(1+x/\sqrt{r})\rd x,\\
I_3(r;m) &= \frac{1}{r^{1/2}}\int_0^6 x^2 |\tilde{g}((1+x/\sqrt{r})^{m+1})|f_r(1+x/\sqrt{r})\rd x.
\end{align}
In Sections \ref{sec:I1I2I3} and \ref{sec:I4},
we will show that
\begin{equation}
 I_1(r;m)\geq I_1^*,\quad I_2(r;m)\leq I_2^*,\quad I_3(r;m)\leq I_3^*,\quad I_4(r;m)\leq I_4^*,
\end{equation}
for all  $ r\geq r_*>36$ and $m\geq c_1\sqrt{r}$, and where
\begin{align}
I_1^* &= \frac{\exp(-1/12r_*)}{\sqrt{2\pi}}\int_{-6}^{c_2}
x^2 \frac{1-2\exp(c_1 x)}{(1-\exp(c_1 x))^2}
\frac{(1+x/\sqrt{r_*})^{r_*}}{\exp( \sqrt{r_*}x)}
\rd x, \\
 I_2^* &=
\frac{1}{\sqrt{2\pi}}\int_{c_2}^0 x^2 
\frac{2\exp(c_1 x)-1}{(1-\exp(c_1 x))^2}
\exp\left(-\frac{x}{\sqrt{r_*}} - \frac{r_*-1}{r_*}\frac{x^2}{2}\right)\rd x, \\
 I_3^* &= \frac{1}{\sqrt{2\pi}}
\int_{0}^6 x^2 
\frac{2(1+x/\sqrt{r_*})^{c_1\sqrt{r_*}}-1}{((1+x/\sqrt{r_*})^{c_1\sqrt{r_*}}-1)^2}
\frac{(1+x/\sqrt{r_*})^{r_*}}{\exp( \sqrt{r_*}x)}
\rd x, \\
I_4^* &=
\frac{2}{c_1^2}\exp\left(-\frac{1}{2}\frac{(6-1/\sqrt{r_*})^2}{1+6/\sqrt{r_*}}\right).
\end{align}
In particular, we have
\begin{equation}\label{I1I2I3I4}
 I_1^* - I_2^* - I_3^* - I_4^* > 0.001 \ \text{for }r_* =20000
\end{equation}
and hence 
\begin{equation}
\Delta(2r+2;m)=R(0,\hat{\theta}_{\JSJS})- R(0,\hat{\theta}(X;m)) \geq 0
\end{equation}
for $ r\geq r_*$ and $m\geq c_1\sqrt{r}$.
For small values of $r$, 
\begin{equation}\label{p*}
\Delta(2r+2;m) > 0 \ \text{ for }1/2\leq r < r_*\text{ and }m\geq c_1\sqrt{r},
\end{equation}
can be shown individually and numerically.
The R code for \eqref{I1I2I3I4} and \eqref{p*} is available at \url{https://github.com/yuzo-maruyama/quarto_stats/blob/main/finite_sum_JS.ipynb}.
Therefore we complete the proof of Theorem \ref{thm:main}.

\subsection{$I_1$, $I_2$ and $I_3$}
\label{sec:I1I2I3}
\subsubsection*{A lower and upper bound of $f_r(u)$}
Recall $f_r(u)$ is given by \eqref{fq}.
Then we have
\begin{equation}\label{fqfq}
 f_r(1+x/\sqrt{r}) = \frac{r^r e^{-r}}{\Gamma(r)} (1+x/\sqrt{r})^{r-1} e^{-\sqrt{r}x},
\end{equation}
and
\begin{equation}
\log\{ (1+x/\sqrt{r})^{r-1} e^{-\sqrt{r}x}\}
= (r-1)\log(1+x/\sqrt{r})-\sqrt{r}x = h(x;r) - \log(1+x/\sqrt{r}),
\end{equation}
where
\begin{equation}\label{hxq}
h(x;r) = r\log(1+x/\sqrt{r}) - \sqrt{r}x.
\end{equation}
By Part \ref{bunbo.2} of Lemma \ref{lemma:log}, for a fixed $x > 0$, $h(x;r)$ is monotonically decreasing in $r$.
Therefore, when $r \ge r_*$, we have
\begin{equation}\label{h}
\log\{ (1+x/\sqrt{r})^{r-1} e^{-\sqrt{r}x}\} \le h(x;r_*), \quad\text{for }x > 0,
\end{equation}
which is used in the evaluation of $I_3(r;m)$.

By Part \ref{bunbo.2} of Lemma \ref{lemma:log}, for a fixed $x < 0$, $h(x;r)$ is monotonically increasing in $r$.
Therefore, for $r \ge r_*$, we have
\begin{equation}\label{hh}
\log\{ (1+x/\sqrt{r})^{r-1} e^{-\sqrt{r}x}\} \ge h(x;r_*), \quad\text{for } x < 0,
\end{equation}
which is used in the evaluation of $I_1(r;m)$.

To bound $I_2(r;m)$ from above for a fixed $x < 0$, we use
\begin{align}\label{hhh}
\log\{ (1+x/\sqrt{r})^{r-1} e^{-\sqrt{r}x}\}
& < (r-1)\left(\frac{x}{\sqrt{r}} - \frac{x^2}{2r}\right) - \sqrt{r}x\\
&= -\frac{x}{\sqrt{r}} - \frac{r-1}{r}\frac{x^2}{2}\\
&\le -\frac{x}{\sqrt{r_*}} - \frac{r_*-1}{r_*}\frac{x^2}{2},
\end{align}
for $r \ge r_*$, 
where the first inequality in \eqref{hhh} follows from Part \ref{bunshi.1} of Lemma \ref{lemma:log}.

For the constant term in \eqref{fqfq}, Stirling's explicit bounds yield
\begin{equation}\label{Stirling}
\frac{\sqrt{r}}{\sqrt{2\pi}} \exp(-1/12r) \le \frac{r^r e^{-r}}{\Gamma(r)} \le 
\frac{\sqrt{r}}{\sqrt{2\pi}} \exp(-1/\{12r+1\}).
\end{equation}

\subsubsection{A lower bound of $I_1$}
In the integration interval $x \in (-6, c_2)$ of $I_1(r;m)$ 
(where $c_2 = -\log 2/c_1$ by \eqref{u1}), we have
\begin{equation}\label{futou.1}
 (1+x/\sqrt{r})^{m+1}
= \{(1+x/\sqrt{r})^{\sqrt{r}}\}^{(m+1)/\sqrt{r}}
\le \{e^x\}^{(m+1)/\sqrt{r}} \le \exp(c_1 x).
\end{equation}
The first inequality follows from the monotonicity of $(1+x/\sqrt{r})^{\sqrt{r}}$ as a function of $\sqrt{r}$, and the second inequality follows from $0 < e^x < 1$.
Furthermore, since $x \le -\log 2/c_1$,
\begin{equation}
 (1+x/\sqrt{r})^{m+1} \le \exp(c_1 x) \le \frac{1}{2}.
\end{equation}
Since $\tilde{g}(y)$ defined in \eqref{tildegy} is monotonically decreasing in $y$ for $y < 1/2$, the integrand of $I_1(r;m)$ satisfies
\begin{equation}
 \tilde{g}((1+x/\sqrt{r})^{m+1})
\ge \tilde{g}(\exp(c_1 x)) = \frac{1-2\exp(c_1 x)}{(1-\exp(c_1 x))^2},
\end{equation}
yielding the lower bound:
\begin{equation}\label{I1I1}
  I_1(r;m) \ge \frac{1}{r^{1/2}}\int_{-6}^{c_2}x^2 \frac{1-2\exp(c_1 x)}{(1-\exp(c_1 x))^2}
f_r(1+x/\sqrt{r})\rd x, 
\end{equation}
for all $r\geq r_*$ and $m\geq c_1\sqrt{r}$.
From \eqref{hh}, \eqref{Stirling}, and \eqref{I1I1}, we have
\begin{equation}
I_1(r;m) \ge \frac{\exp(-1/12r_*)}{\sqrt{2\pi}}\int_{-6}^{c_2}
x^2 \frac{1-2\exp(c_1 x)}{(1-\exp(c_1 x))^2}
\frac{(1+x/\sqrt{r_*})^{r_*}}{\exp( \sqrt{r_*}x)}
\rd x,
\end{equation}
for all $r\geq r_*$ and $m\geq c_1\sqrt{r}$.

\subsubsection{An upper bound of $I_2$}
In the integration interval $x \in (\sqrt{r}(u_0-1), 0)$ of $I_2(r;m)$, 
by definition \eqref{u0} of $u_0$, we have
\begin{equation}
\frac{1}{2} \le (1+x/\sqrt{r})^{m+1} \le 1.
\end{equation}
The function $|\tilde{g}(y)|$ is monotonically increasing in $y$ for $1/2 < y < 1$.
For $x \in (\sqrt{r}(u_0-1), 0)$, similarly to \eqref{futou.1}, we have
\begin{equation}\label{futou.1.5}
 (1+x/\sqrt{r})^{m+1}
= \{(1+x/\sqrt{r})^{\sqrt{r}}\}^{(m+1)/\sqrt{r}}
\le \{e^x\}^{(m+1)/\sqrt{r}} \le \exp(c_1 x),
\end{equation}
and therefore
\begin{equation}
| \tilde{g}((1+x/\sqrt{r})^{m+1})|
\le |\tilde{g}(\exp(c_1 x))|
= \frac{2\exp(c_1 x)-1}{(1-\exp(c_1 x))^2}.
\end{equation}
Regarding the lower endpoint $\sqrt{r}(u_0-1)$ of the integration interval, it follows from \eqref{uhen} and \eqref{u1} that $c_2 < \sqrt{r}(u_0-1)$.
Since $2\exp(c_1 x) - 1 \ge 0$ follows $x\in(c_2, 0)$, we can widen the integration range to $(c_2, 0)$ to obtain the upper bound:
\begin{equation}\label{I2I2}
 I_2(r;m) \le 
\frac{1}{r^{1/2}}\int_{c_2}^0 x^2 
\frac{2\exp(c_1 x)-1}{(1-\exp(c_1 x))^2}
f_r(1+x/\sqrt{r})\rd x,
\end{equation}
for all $r\geq r_*$ and $m\geq c_1\sqrt{r}$.
From \eqref{hhh}, \eqref{Stirling}, and \eqref{I2I2}, we have
\begin{equation}
 I_2(r;m) \le 
\frac{1}{\sqrt{2\pi}}\int_{c_2}^0 x^2 
\frac{2\exp(c_1 x)-1}{(1-\exp(c_1 x))^2}
\exp\left(-\frac{x}{\sqrt{r_*}} - \frac{r_*-1}{r_*}\frac{x^2}{2}\right)\rd x,
\end{equation}
for all $r\geq r_*$ and $m\geq c_1\sqrt{r}$.

\subsubsection{An upper bound of $I_3$}
In the integration interval $x \in (0, 6)$ of $I_3(r;m)$, for all $r \ge r_*$, we have
\begin{equation}
(1+x/\sqrt{r_*})^{c_1\sqrt{r_*}}
\le (1+x/\sqrt{r})^{c_1\sqrt{r}}
\le (1+x/\sqrt{r})^{m+1},
\end{equation}
where the first inequality follows from the monotonicity of $(1+x/\sqrt{r})^{\sqrt{r}}$ as a function of $\sqrt{r}$, and the second inequality follows from $1+x/\sqrt{r} \ge 1$.
Since $|\tilde{g}(y)|$ is monotonically decreasing in $y$ for $y > 1$,
\begin{equation}
| \tilde{g}((1+x/\sqrt{r})^{m+1})|
\le |\tilde{g}((1+x/\sqrt{r_*})^{c_1\sqrt{r_*}})|
= \frac{2(1+x/\sqrt{r_*})^{c_1\sqrt{r_*}}-1}{((1+x/\sqrt{r_*})^{c_1\sqrt{r_*}}-1)^2}.
\end{equation}
Therefore, we have
\begin{equation}\label{I3I3}
 I_3(r;m) \le \frac{1}{r^{1/2}}\int_{0}^6 x^2 
\frac{2(1+x/\sqrt{r_*})^{c_1\sqrt{r_*}}-1}{((1+x/\sqrt{r_*})^{c_1\sqrt{r_*}}-1)^2}
f_r(1+x/\sqrt{r})\rd x, 
\end{equation}
for all $r\geq r_*$ and $m\geq c_1\sqrt{r}$.
From \eqref{Stirling}, \eqref{I3I3}, and \eqref{h}, we have
\begin{equation}
 I_3(r;m) \le \frac{1}{\sqrt{2\pi}}
\int_{0}^6 x^2 
\frac{2(1+x/\sqrt{r_*})^{c_1\sqrt{r_*}}-1}{((1+x/\sqrt{r_*})^{c_1\sqrt{r_*}}-1)^2}
\frac{(1+x/\sqrt{r_*})^{r_*}}{\exp( \sqrt{r_*}x)}
\rd x,
\end{equation}
for all $r\geq r_*$ and $m\geq c_1\sqrt{r}$.

\subsection{An upper bound of $I_4$}
\label{sec:I4}
Recall 
\begin{equation}
I_4(r;m) = r \int_{1+6/\sqrt{r}}^\infty |g(u;m)|f_r(u)\rd u.
\end{equation}
For $u>1$, $|g(u;m)|$ given by \eqref{gu.0} is
\begin{equation}
 |g(u;m)| = \frac{2u^{m+1}-1}{\{[m+1]_u\}^2}.
\end{equation}
By Part \ref{lem:q-number-2} of Lemma \ref{lem:q-number} below,
we have
\begin{equation}
 |g(u;m)| \le \frac{2u^{m+1}}{\{(m+1) u^{m/2}\}^2} \le \frac{2u}{(m+1)^2} \le \frac{2}{c_1^2}\frac{u}{r}
\end{equation}
and hence
\begin{align}
I_4(r;m) &\le \frac{2}{c_1^2}\int_{1+6/\sqrt{r}}^\infty u f_r(u)\rd u\\
&= \frac{2}{c_1^2}\int_{1+6/\sqrt{r}}^\infty u \frac{r^r}{\Gamma(r)}u^{r-1}e^{-ru}\rd u\\
&= \frac{2}{c_1^2}\int_{1+6/\sqrt{r}}^\infty \frac{r^{r+1}}{\Gamma(r+1)}u^{r}e^{-ru}\rd u\\
&= \frac{2}{c_1^2}\mathrm{Pr}(\tilde{U} \ge 1+6/\sqrt{r}),
\end{align}
where $\tilde{U} \sim \mathrm{Ga}(r+1, 1/r)$.
Chernoff's concentration inequality for the gamma distribution 
$Y\sim\mathrm{Ga}(\alpha, \beta)$ states that for any $\delta > 0$,
\begin{equation}
 \mathrm{Pr}(Y \ge (1+\delta)\alpha\beta) \le \exp\left(-\alpha\{\delta - \log(1+\delta)\}\right).
\end{equation}
Applying the inequality from Part \ref{bunshi.1} of Lemma \ref{lemma:log} to $\log(1+\delta)$,
\begin{equation}
 \log(1+\delta) \le \delta\frac{1+\delta/2}{1+\delta} = \delta - \frac{\delta^2}{2(1+\delta)},
\end{equation}
we obtain
\begin{equation}
 \mathrm{Pr}(Y \ge (1+\delta)\alpha\beta) \le \exp\left(-\frac{\alpha\delta^2}{2(1+\delta)}\right).
\end{equation}
Substituting $\alpha = r+1$, $\beta = 1/r$, and $\displaystyle \delta = \frac{6/\sqrt{r}-1/r}{1+1/r}$, we get
\begin{equation}
 \mathrm{Pr}(\tilde{U} \ge 1+6/\sqrt{r}) \le
\exp\left(-\frac{1}{2}\frac{(6-1/\sqrt{r})^2}{1+6/\sqrt{r}}\right).
\end{equation}
For all $r\geq r_*$ and $m\geq c_1\sqrt{r}$, we have
\begin{equation}
I_4(r;m) < \frac{2}{c_1^2}\mathrm{Pr}(\tilde{U} \ge 1+6/\sqrt{r}) \le
\frac{2}{c_1^2}\exp\left(-\frac{1}{2}\frac{(6-1/\sqrt{r_*})^2}{1+6/\sqrt{r_*}}\right).
\end{equation}

\subsection{The lower bound of $m$}
\label{sec:0.94}
By the calculation in previous sections,
$\lim_{r\to\infty}\Delta(2r+2;a\sqrt{r})$ for $a>0$ is given by
\begin{equation}
\lim_{r\to\infty}\Delta(2r+2;a\sqrt{r})= \int_{-\infty}^\infty\frac{1-2e^{ax}}{(e^{ax}-1)^2}\frac{1}{\sqrt{2\pi}}e^{-x^2/2}\rd x.
\end{equation}
The value of $a$ that yields $\lim_{r\to\infty}\Delta(2r+2;a\sqrt{r})=0$ is numerically found to be $1.3283\approx 0.9393\sqrt{2}$. 
The constant $0.95$ is slightly larger than the approximated value $0.9393$.

\subsection{Lemma}
\begin{lemma}\label{lem:q-number}
\begin{enumerate}
 \item\label{lem:q-number-1} For $n>1$, $q (\rd/\rd q)[n]_q$ is strictly increasing in $q$.
 \item\label{lem:q-number-0} $ [n]_q$ is strictly increasing in $n$.
 \item\label{lem:q-number-2} For $n>1$, $[n]_q\geq n q^{(n-1)/2}$.
\end{enumerate} 
\end{lemma}
\begin{proof}{}
[Part \ref{lem:q-number-1}] We start with the integral expression 
\begin{equation}\label{int.q-number}
 [n]_q=n\int_0^1 \{1+s(q-1)\}^{n-1}\rd s.
\end{equation}
Then we have
\begin{equation}
 q \frac{\rd}{\rd q}[n]_q=
qn\frac{\rd}{\rd q}\int_0^1 \{1+s(q-1)\}^{n-1}\rd s
=n(n-1)\int_0^1 h(q;s,n)\rd s,
\end{equation}
where
\begin{equation}
 h(q;s,n) = qs\{1+s(q-1)\}^{n-2}.
\end{equation}
Further 
\begin{equation}
 \frac{\rd}{\rd q}\left(q \frac{\rd}{\rd q}[n]_q\right)
=n(n-1)\int_0^1 \frac{\rd}{\rd q}h(q;s,n)\rd s,
\end{equation}
where
\begin{align}
\frac{\rd}{\rd q}h(q;s,n) &= s\{1+s(q-1)\}^{n-2} +q(n-2)s^2\{1+s(q-1)\}^{n-3} \\
&= s\{1+s(q-1)\}^{n-3} \{1+s(q-1)+s(n-2)q\} \\ 
&= s\{1+s(q-1)\}^{n-3} \{(1-s)+s(n-1)q\},
\end{align}
which is non-negative for $q>0$, $n>1$, and $0\le s\le1$.
Then we completed the proof of Part \ref{lem:q-number-1}.

[Part \ref{lem:q-number-0}]
For $q\neq 1$, the derivative of $ [n]_q$ with respect to $n$
is
\begin{equation}
 \frac{\rd}{\rd n}[n]_q=\frac{(\log q) q^n}{q-1}>0.
\end{equation}
Further we have $\lim_{q\to 1}(\rd /\rd n)[n]_q=1$ and $[n]_1=n$, completing the proof of
Part \ref{lem:q-number-0}.

[Part \ref{lem:q-number-2}] In the integrand of \eqref{int.q-number}, 
by a weighted AM-GM inequality, we have
\begin{equation}
 1+s(q-1)=sq+(1-s)\geq q^s
\end{equation}
and hence
\begin{equation}
 [n]_q=n\int_0^1 \{1+s(q-1)\}^{n-1}\rd s\geq n\int_0^1 q^{s(n-1)}\rd s.
\end{equation}
Further Jensen's inequality gives
\begin{align}
 \int_0^1 q^{s(n-1)}\rd s&=\int_0^1\exp(s \{(n-1)\log q\})\rd s
\geq
\exp\left(\int_0^1 s\rd s \{(n-1)\log q\}\right)\\ &=\exp((n/2-1/2)\log q)=q^{(n-1)/2},
\end{align}
which completes the proof of Part \ref{lem:q-number-2}.
\end{proof}

\begin{lemma}\label{lemma:log}
\begin{enumerate}
 \item\label{bunshi.1} For $x > -1$,
\begin{empheq}[left={\log(1+x)\empheqlbrace}]{alignat=2}
  & < \frac{x(x+2)}{2(x+1)} &\qquad &\text{if $x > 0$,} \label{em1}\\
  & > \frac{x(x+2)}{2(x+1)} &       &\text{if $-1 < x < 0$}. \label{em2}
\end{empheq}

\item\label{bunbo.1} 
For $x > -1$,
\begin{empheq}[left={\log(1+x)\empheqlbrace}]{alignat=2}
  & > x - \frac{x^2}{2} &\qquad &\text{if $x > 0$,} \label{em3}\\
  & < x - \frac{x^2}{2} &       &\text{if $-1 < x < 0$.} \label{em4}
\end{empheq}

\item \label{bunbo.2}  
\begin{enumerate}
\item For fixed $t > 0$, $(1+t/\sqrt{r})^{r}e^{-t\sqrt{r}}$ 
is monotonically non-increasing in $r$. 
\item For fixed $-\sqrt{r} < t < 0$, $(1+t/\sqrt{r})^{r}e^{-t\sqrt{r}}$ 
is monotonically non-decreasing in $r$. 
\end{enumerate}
\end{enumerate}
\end{lemma}

\begin{proof}
Part \ref{bunshi.1}: Let $x > 0$. 
The area enclosed by the curve $f(t) = 1/t$, the $t$-axis, and the vertical lines $t = 1$ and $t = x+1$ is given by $\log(1+x)$.
For the trapezoid with vertices at 
\begin{equation}
\{ (1,f(1)), \ (1,0), \ (1+x, 0), \ (1+x,f(1+x))\}, 
\end{equation}
the area is $x(x+2)/\{2(1+x)\}$.
Since $f(t)$ is convex, inequality \eqref{em1} follows.

Let $-1 < x < 0$. Then
\begin{equation}
-\log(1+x) = \log\left(1+\frac{-x}{1+x}\right)
\le \frac{\{-x/(1+x)\}\{-x/(1+x)+2\}}{2\{-x/(1+x)+1\}} = -\frac{x(x+2)}{2(x+1)},
\end{equation}
where the inequality follows from \eqref{em1}. Inequality \eqref{em2} then follows.

\medskip

Part \ref{bunbo.1}: Let $f(x) = \log(1+x) - \{x - (1/2)x^2\}$. Then
\begin{equation}
 f'(x) = \frac{1}{1+x} - 1 + x = \frac{1-(x+1)+x(x+1)}{x+1} = \frac{x^2}{1+x}.
\end{equation}
Since $f(0) = 0$ and $f'(x) > 0$ for $x > -1$, inequalities \eqref{em3} and \eqref{em4} follow.

\medskip

Part \ref{bunbo.2}: 
We have
\begin{align}
 \frac{\rd}{\rd r} \log\{(1+t/\sqrt{r})^{r}e^{-t\sqrt{r}}\} &= \log(1+t/\sqrt{r})
+ r\frac{-(1/2)r^{-3/2}t}{1+t/\sqrt{r}}
- \frac{1}{2}r^{-1/2}t \\
&= -\frac{t/\sqrt{r}(2+t/\sqrt{r})}{2(1+t/\sqrt{r})}
+ \log(1+t/\sqrt{r}).
\end{align}
By \eqref{em1} and \eqref{em2} of Part \ref{bunshi.1}, we have
\begin{equation}
 \frac{\rd}{\rd r} \log\{(1+t/\sqrt{r})^{r}e^{-t\sqrt{r}}\}
\begin{cases}
\le 0 &\text{if } t > 0,\\ 
\ge 0 &\text{if } -\sqrt{r} < t < 0,
\end{cases}
\end{equation}
which completes the proof of Part \ref{bunbo.2}.
\end{proof}

\end{document}